\documentclass[11pt]{amsart}
\usepackage[utf8]{inputenc}
\usepackage[T1]{fontenc} 
\usepackage{lmodern}    
\usepackage{newtxtext,newtxmath, mathtools}
\usepackage{listings}
\usepackage{tikz-cd} 
\usepackage[all,2cell,dvips]{xy}
\usepackage[pagebackref, colorlinks=true, pdfstartview=FitV, linkcolor=red, citecolor=blue, urlcolor=blue]{hyperref}
\usepackage{cleveref}
\usepackage[margin=1in]{geometry}
\usepackage{parskip}
\usepackage{enumitem}
\newtheorem{thm}{Theorem}[section]
\newtheorem{lemma}[thm]{Lemma}
\newtheorem{rmk}[thm]{Remark}
\newtheorem{prop}[thm]{Proposition}
\newtheorem{cor}[thm]{Corollary}
\newtheorem{defn}[thm]{Definition}
\newtheorem{question}[thm]{Question}

\newcommand{\im}{\mbox{\rm Im}}

\newcommand{\Tor}{\mbox{\rm Tor}}

\newcommand{\depth}{\mbox{\rm depth}}

\newcommand{\hgt}{\mbox{\rm ht}}

\newcommand{\be}{\mbox{{\bf e}}}

\begin{document}

\title[On Golod Subdeterminantal Ideals]{On Golod Subdeterminantal Ideals}

\author{Omkar Javadekar}
\address{Chennai Mathematical Institute, Siruseri, Tamilnadu 603103. India}
\email{omkarj@cmi.ac.in}



\begin{abstract}
Let $X=(x_{ij})_{m\times n}$ be a matrix of indeterminates and let  
$S=\mathbb{k}[x_{ij} \mid 1\le i\le m,\ 1\le j\le n]$ be a polynomial ring over an infinite field $\mathbb{k}$.  
Let $I$ be an ideal generated by a subset of the set of all $2\times2$ minors of $X$.  
We show that the quotient ring $S/I$ is Golod if and only if $I=I_2(Y)$ for some $2\times \ell$ or $\ell\times2$ submatrix $Y$ of $X$.  
In fact, we prove that Golodness of $S/I$ is equivalent to the triviality of the product on the Koszul homology of $S/I$ and to $I$ having a linear resolution.
Along the way, we also prove a result on the non-Golodness of tensor products of rings under certain conditions.
\end{abstract}

\maketitle

\section{Introduction}

Let $S=\mathbb{k}[x_1, \ldots, x_n]$ be a standard graded polynomial ring over a field $\mathbb{k}$, $I\subseteq (x_1, \ldots, x_n)^2$ a homogeneous ideal, and $R=S/I$. The power series $\mathcal{P}_{\mathbb{k}}^R(z)= \sum_{i \geq 0} \dim_{\mathbb{k}} \left(\Tor_i^R(\mathbb{k},\mathbb{k})\,\right)z^i$ is called the \textit{Poincar\'e series of $\mathbb{k}$ over $R$}. Due to a result of Serre \cite{Se65}, there is a term-by-term inequality
\begin{align*}
    \mathcal{P}_{\mathbb{k}}^R(z) \preceq \dfrac{(1+z)^n}{1-\sum\limits_{i \geq 1}\dim_{\mathbb{k}}\left(\Tor_i^S(R,\mathbb{k})\right)\, z^{i+1}}. 
\end{align*}
In \cite{Go62}, Golod proved that that the equality holds if and only if all Massey operations on the homology of the Koszul complex $\mathcal K(x_1, \ldots, x_n; R)$ are trivial; in this case, the ring $R$ is called \textit{Golod}. 

Detecting Golodness is a rather hard problem. For instance, we do not have a complete characterization of Golodness of $S/I$, even when $I$ is a monomial ideal. Attempts have been made to characterize Golodness of monomial quotients, especially those arising from combinatorial objects, such as Stanley--Reisner rings of simplicial complexes. Golod ideals over polynomial rings in at most 4 variables have also been characterized. For a non-exhaustive list of works on detecting Golodness and related criteria, we refer the interested reader to 
\cite{Ah19,ANFY17,BJ07,DS22,DS16,Fr18,GTSW16,HH13,HRW99,IK18,IK23,IK24,Ka16,Ka17,FW14,Va22,Zo24}
 and the references therein.

When $R$ is Golod, the triviality of all Massey operations forces the Koszul homology
$\mathcal H(\mathcal K(x_1, \ldots, x_n; R))$ to have trivial product. However, in general, 
this condition is not sufficient for Golodness, even for monomial quotients, as the examples due to De Stefani \cite{DS16} and Katth\"an \cite{Ka17} show. 
In some special cases, the triviality of the product on Koszul homology is indeed equivalent to Golodness. 
For instance, if $I$ is a monomial ideal generated in degree two, then combining the results proved in \cite{BJ07,HerzogHibiZheng, HRW99}, 
it follows that $S/I$ is Golod if and only if $I$ has a linear resolution if and only if the product on the Koszul homology is trivial. 
The main result of our paper is in the same spirit. 

We consider ideals generated by subsets of the set of all $2 \times 2$ minors of a matrix $X=(x_{ij})_{m \times n}$ of indeterminates, which we call \textit{$2$-subdeterminantal ideals of $X$} (see Definition \ref{def:subdet-ideal}). 
In the literature, these ideals have also been referred to as \textit{generalized binomial edge ideals} or \textit{binomial edge ideals of a pair of graphs} (see \cite{EHHQ14}). 
Due to this identification, these ideals are of interest in combinatorial commutative algebra, and several homological properties of the associated quotient rings, such as the Castelnuovo--Mumford regularity, Koszulness, existence of quadratic Gr\"obner bases, etc.~have been studied in recent years (see, e.g.,~\cite{BEI18, EHHQ14, JK25, LMMP26}). 
The class of ideals we consider also includes the classical determinantal ideals as a special case (see \cite{BV88}). 
Our aim in this article is to provide a complete characterization of the Golodness of these ideals. 
Our main theorem is as follows.

\begin{thm}\label{intro-main-thm}
    Let $\mathbb{k}$ be an infinite field and $X=(x_{ij})_{m \times n}$ be a matrix of indeterminates. Suppose that $S$ denotes the polynomial ring $\mathbb{k}[x_{ij} \mid 1 \leq i \leq m, 1\leq j \leq n]$, and $I$ is a $2$-subdeterminantal ideal of $X$. Then the following are equivalent:
    \begin{enumerate}
        \item[{\rm (i)}] The ring $S/I$ is Golod.
        \item[{\rm (ii)}] $I$ has a linear resolution over $S$.
        \item[{\rm (iii)}] $I=I_2(Y)$ for some $2 \times \ell$ or $\ell \times 2$ submatrix $Y$ of $X$.
        \item[{\rm (iv)}] The product on the Koszul homology $\mathcal H(\mathcal K (x_{11}, \ldots,x_{mn}; S/I))$ is trivial.
    \end{enumerate} 
\end{thm}
In particular, Theorem \ref{intro-main-thm} says that the binomial edge ideal of a pair of graphs $(G_1,G_2)$ is Golod if and only if both $G_1$ and $G_2$ are complete and one of them is equal to $K_2$.

Our proof of Theorem \ref{intro-main-thm} proceeds in several steps. We begin with the case where $X$ is a $2 \times n$ matrix, and then analyze the $3 \times 3$ and $3 \times n$ cases. These preliminary cases provide the foundation for the main argument. Along the way, we prove a general result on the non-Golodness of certain tensor products, which allows us to reduce the problem to a simpler setting. This result is in the same spirit as the classical statement that a Golod complete intersection must be a hypersurface. A key ingredient in the proof of our main theorem is the fact that for a Gorenstein local ring, the homology algebra of its Koszul complex is a Poincar\'e algebra, as shown by Avramov and Golod \cite{AG71}. Assuming that $\mathbb{k}$ is infinite, the same result holds for standard graded Gorenstein $\mathbb{k}$-algebras, which is the version we shall use.

The article is organized as follows. In Section~2, we recall the necessary definitions and set up the notation. We also collect relevant known results and extract some immediate consequences. 
Section~3 studies the Golod property of certain tensor products and proves a more general result for $2$-subdeterminantal ideals, which will be used in subsequent reductions. 
In Section~4, we prove several preparatory results and establish the main theorem. We conclude with remarks and questions aimed at possible generalizations of our results, including higher-order minors, other classes of matrices, and related types of ideals.

\section*{Acknowledgments}
The author thanks Aryaman Maithani for his help with some  \texttt{Macaulay2}~\cite{M2} computations. The author is also grateful to Mugdha Pokharanakar and Abhiram Subramanian for their comments that improved the exposition. 
The author acknowledges support from a Postdoctoral Fellowship at the Chennai Mathematical Institute and additional support from the Infosys Foundation.

\section{Preliminaries}\label{sec:prelim}

Throughout this article, $\mathbb{k}$ denotes an infinite field. 
Let $S$ be a standard graded polynomial ring over $\mathbb{k}$, $I$ a homogeneous ideal of $S$, and $R=S/I$. 
Let $f_1,\ldots,f_d$ be a sequence of homogeneous elements in $R$. The \emph{Koszul complex} $\mathcal{K}(f_1,\ldots,f_d;R)$ is the graded complex of free $R$-modules
\[
\mathcal{K}(f_1,\ldots,f_d;R):\quad 0\to \wedge^d R^d \xrightarrow{\partial_d} \wedge^{d-1} R^d \xrightarrow{\partial_{d-1}} \cdots \xrightarrow{\partial_2} R^d \xrightarrow{\partial_1} R \to 0,
\]
where, $\partial_1(\be_i)=f_i$, and  for $i_1<\cdots<i_r$, the differential is given by
\[
\partial_r(\be_{i_1}\wedge\cdots\wedge \be_{i_r}) = \sum_{s=1}^r (-1)^{s-1} f_{i_s} \be_{i_1}\wedge\cdots\widehat{\be_{i_s}}\cdots\wedge \be_{i_r}.
\]
We may sometimes abuse the notation and omit $\wedge$. The \emph{$i^{th}$ Koszul homology} is then defined as $$\mathcal{H}_i(\mathcal{K}(f_1,\ldots,f_d;R)) := \ker(\partial_i)/\im(\partial_{i+1}).$$ It is well-known that $\mathcal H(\mathcal K(f_1,\ldots,f_d;R))$ inherits the product structure coming from the exterior algebra structure on the Koszul complex. 
\begin{defn}
We say that the Koszul homology $\mathcal H(\mathcal K(f_1,\ldots,f_d;R))$ has trivial product if
\[
\mathcal H_{\geq 1}(\mathcal K(f_1,\ldots,f_d;R)) \cdot 
\mathcal H_{\geq 1}(\mathcal K(f_1,\ldots,f_d;R)) = 0.
\]
If this product is nonzero, we say that the Koszul homology has a nontrivial product.
\end{defn}

\begin{defn}
    Given a finitely generated graded $R$-module $M$, its $i^{th}$ Betti number $\beta_i^R(M)$, and the $(i,j)^{th}$ graded Betti number $\beta_{i,j}^R(M)$ are defined as 
$$\beta_{i}^R(M) = \dim_{\mathbb{k}}\left(\Tor ^R_i(M, \mathbb{k})\right) {{\quad \quad\text{and} }}\quad\quad\beta_{i,j}^R(M) = \dim_{\mathbb{k}}\left(\Tor ^R_i(M, \mathbb{k})\right)_j.$$
The series $\mathcal P_M^R(z)= \sum_{i \geq 0} \beta_i^R(M)z^i$ is called the \emph{Poincar\'e series of $M$ over $R$.}
If $M$ is generated in degree $d$, then we say that \emph{$M$ has a linear resolution over $R$} if $\beta_{i,j}^R(M)\neq 0$ implies $j = d+i$.
\end{defn}

Given a homogeneous ideal $I$ in a polynomial ring $S$, we let $\mu(I)$ denote the cardinality of a minimal generating set of $I$. In other words, $\mu(I)=\beta_0^S(I)=\beta_1^S(S/I)$.
\begin{rmk}\label{rmk:Golod-def} {\rm In this remark, we recall some known facts about Golodness which we shall use later. 
    Let $S=\mathbb{k}[x_1,\ldots, x_n]$, $I$ be a homogeneous ideal of $S$, and $R=S/I$.
\begin{enumerate}[label={\rm (\alph*)}]
    \item  Since $\beta_i^S(R)=\dim_{\mathbb{k}}\left(\Tor ^S_i(R, \mathbb{k})\right)$ and the Koszul complex $\mathcal K(x_1, \ldots, x_n; S)$ gives a minimal free resolution of $\mathbb{k}$ over $S$, the term-by-term inequality due to Serre mentioned in Section~1 the can be rewritten as 
   $$\mathcal P^R_{\mathbb{k}}(z) \preceq \dfrac{(1+z)^n}{1-z(\mathcal P^S_R(z)-1)}.$$
   The ring $R$ is Golod if and only if the above inequality becomes an equality.
    \item If $R$ is Golod, then the Koszul homology $\mathcal H(\mathcal K(x_1, \ldots, x_n; R))$ has trivial product (see \cite[Remark 5.2.1]{Av98}).
    
    \item If $x$ is a linear nonzerodivisor on $R$, then $R$ is Golod if and only if $R/( x )$ is Golod (see \cite[Proposition 5.2.4]{Av98}).
    \end{enumerate}
}
\end{rmk}

Let $X=(x_{ij})_{m \times n}$ be a matrix of indererminates. From now on, we use $\mathbb{k}[X]$ to denote the polynomial ring $\mathbb{k}[x_{ij} \mid 1 \leq i \leq m, 1 \leq j \leq n]$. 

\begin{defn}\label{def:subdet-ideal}
    Given any $t\geq 1$, an ideal $I$ of $\mathbb{k}[X]$ is called a \emph{$t$-subdeterminantal ideal of $X$}, if $I$ is generated by a subset of the set of all $t \times t$ minors of $X$. 
    Also, we use $I_t(X)$ to denote the ideal generated by all $t \times t$ minors of $X$. 
\end{defn}

\begin{rmk}\label{rmk:det-ideal-height}{\rm Let $X=(x_{ij})_{m \times n}$, $S=\mathbb{k}[X]$, and $I$ a $t$-subdeterminantal ideal of $X$. If $\mu(I)\geq 2$, then $\hgt(I)\geq 2$. This gives $\depth(S/I)\leq \dim(S/I)\leq mn-2$. Hence, by the Auslander--Buchsbaum formula (see \cite[Theorem 1.3.3]{BH93}), we get that $\beta_i^S(S/I)\neq 0$ for $0 \leq i \leq 2$. But since $\beta_i^S(S/I)$ is precisely the $\mathbb{k}$-vector space dimension of $\mathcal H_i(\mathcal{K}(x_{11}, \ldots, x_{mn}; S/I))$, we obtain $\mathcal H_1(\mathcal{K}(x_{11}, \ldots, x_{mn}; S/I)) \neq 0$ and $\mathcal H_2(\mathcal{K}(x_{11}, \ldots, x_{mn}; S/I)) \neq 0$. 

    This fact will be used crucially in several results in Section~4. 
}\end{rmk}

\begin{rmk}\label{rmk:grading}
     {\rm Suppose $X=(x_{ij})_{m \times n}$, $t \geq 2$, and let $S=\mathbb{k}[X]$ be a standard graded polynomial ring. Let $I$ be a $t$-subdeterminantal ideal of $X$. 
     \\
We may also look at $S$ as a graded ring, where the grading is defined by setting $\deg(x_{ij})=\epsilon_j$. Then, with respect to this grading, the ideal $I$ is a graded ideal of $S$. Thus, the quotient ring also acquires this grading, and so does the Koszul complex $\mathcal K(x_{11}, \ldots, x_{mn}; S/I)$, with degree zero differential.\\
Note that there is a natural partial order $\leq$ on the degrees. We say that $\sum \alpha_i\epsilon_i \leq \sum \gamma_i\epsilon_i$ if and only if $\alpha_i\leq \gamma_i$ for all $i$.
We shall make use of these facts in the proof of our main theorem. 
}\end{rmk}

\section{Golodness of Tensor Products}

We begin with a characterization of the Golodness of tensor product of certain standard graded $\mathbb{k}$-algebras. The next result may be viewed as an analogue of the classical statement that a Golod complete intersection is necessarily a hypersurface.

\begin{prop}\label{prop:tensor}
    Let \(I_1 \subseteq ( x_1,\ldots, x_n)^2 \subseteq \mathbb{k}[x_1,\ldots, x_n]\) and 
    \(I_2 \subseteq ( y_1,\ldots, y_m)^2 \subseteq \mathbb{k}[y_1,\ldots,y_m]\) be homogeneous ideals 
    with \(I_1 \neq 0\). Set
    \[
    R_1 = \frac{\mathbb{k}[x_1,\ldots, x_n]}{I_1}, \qquad
    R_2 = \frac{\mathbb{k}[y_1,\ldots, y_m]}{I_2}, \qquad
    R = R_1 \otimes_{\mathbb{k}} R_2.
    \]
    Then \(R\) is Golod if and only if \(R_1\) is Golod and \(I_2 = 0\).
\end{prop}

   \begin{proof}
        Let \(S_1 = \mathbb{k}[x_1,\ldots, x_n]\), \(S_2 = \mathbb{k}[y_1,\ldots, y_m]\), and 
    \(S = S_1 \otimes_{\mathbb{k}} S_2 \cong \mathbb{k}[x_1,\ldots, x_n, y_1,\ldots, y_m]\). 
    Then, with \(I = I_1 S + I_2 S\), we have \(R \cong S/I\).

    $(\Longleftarrow)$ Suppose \(R_1\) is Golod and \(I_2 = 0\). Then \(R \cong R_1[y_1,\ldots,y_m]\), 
    and \(y_1,\ldots,y_m\) is a regular sequence on \(R\). By Remark~\ref{rmk:Golod-def}(c), 
    \(R\) is Golod.

    $(\Longrightarrow)$ Assume that \(R\) is Golod. We first show that \(I_2 = 0\). 
    The Golod property of \(R\) gives the following equality for the Poincar\'e series:
    \[
    \mathcal{P}^R_{\mathbb{k}}(z) = \frac{(1+z)^{n+m}}{1 - z\bigl(\mathcal{P}^S_R(z) - 1\bigr)},
    \tag{1}
    \]
    by Remark~\ref{rmk:Golod-def}(a). On the other hand, from \cite[2.3]{Jo09} we have
    \[
    \mathcal{P}^R_{\mathbb{k}}(z) = \mathcal{P}^{R_1}_{\mathbb{k}}(z) \mathcal{P}^{R_2}_{\mathbb{k}}(z),
    \qquad
    \mathcal{P}^S_R(z) = \mathcal{P}^{S_1}_{R_1}(z) \mathcal{P}^{S_2}_{R_2}(z). \tag{2}
    \]
   Again using Remark \ref{rmk:Golod-def}(a), we have the term-by-term inequalities
    \[
    \mathcal{P}^{R_1}_{\mathbb{k}}(z) \preceq 
    \frac{(1+z)^n}{1 - z\bigl(\mathcal{P}^{S_1}_{R_1}(z) - 1\bigr)}, \qquad
    \mathcal{P}^{R_2}_{\mathbb{k}}(z) \preceq 
    \frac{(1+z)^m}{1 - z\bigl(\mathcal{P}^{S_2}_{R_2}(z) - 1\bigr)}.
    \]
    Multiplying these inequalities and using (2), we obtain
    \[
    \mathcal{P}^R_{\mathbb{k}}(z) \preceq 
    \frac{(1+z)^{n+m}}{\bigl(1 - z(\mathcal{P}^{S_1}_{R_1}(z) - 1)\bigr)
                      \bigl(1 - z(\mathcal{P}^{S_2}_{R_2}(z) - 1)\bigr)}. \tag{3}
    \]

    Combining (1) and (3), we deduce
    \[
    \frac{1}{1 - z(\mathcal{P}^S_R(z) - 1)} \preceq
    \frac{1}{\bigl(1 - z(\mathcal{P}^{S_1}_{R_1}(z) - 1)\bigr)
            \bigl(1 - z(\mathcal{P}^{S_2}_{R_2}(z) - 1)\bigr)}. \tag{4}
    \]
    Expanding both sides of (4) as power series in \(z\), we see that the coefficient of \(z^3\) on the left-hand side equals
    \[
    \beta_1^{S_1}(I_1) + \beta_1^{S_2}(I_2) + \mu(I_1)\mu(I_2),
    \]
    while on the right-hand side, it equals
    \[
    \beta_1^{S_1}(I_1) + \beta_1^{S_2}(I_2).
    \]
    Since \(\mu(I_1) > 0\) by hypothesis, inequality (4) forces \(\mu(I_2) = 0\); hence \(I_2 = 0\). 
    Thus, we have \(R=R_1 \otimes_{\mathbb{k}} S_2\). 
    The images of \(y_1,\ldots,y_m\) in \(R\) form a regular sequence, and by 
    Remark~\ref{rmk:Golod-def}(c), the quotient \(R/(y_1,\ldots,y_m) \cong R_1\) is Golod. This completes the proof. 
   \end{proof}

In fact, a stronger statement than the one in Proposition \ref{prop:tensor} holds. Let the setup be as in Proposition~\ref{prop:tensor}. Consider the Koszul complexes $\mathcal K_1\coloneqq \mathcal K(x_1, \ldots, x_n; S_1/I_1)$, $\mathcal K_2\coloneqq \mathcal K(y_1, \ldots, y_m; S_2/I_2)$, and $$\mathcal K\coloneqq \mathcal K_1 \otimes_{\mathbb{k}} K_2= \mathcal K(x_1, \ldots, x_n, y_1, \ldots, y_m; S/I).$$ Applying the K\"unneth theorem (see \cite[Chapter IV]{CE56}), for every integer $k$, we have an isomorphism 

$$ \bigoplus_{i+j=k} \mathcal H_i(\mathcal K_1) \otimes_{\mathbb{k}} \mathcal H_j(\mathcal K_2) \xrightarrow[]{\alpha} \mathcal H_k(\mathcal K),$$
where the map $\alpha$ respects the product structure on the Koszul homology. To be more precise, if $a \in \mathcal H_i(\mathcal K_1)$ and $b \in \mathcal H_j(\mathcal K_2)$, then naturally $a \in \mathcal H_i(\mathcal K)$ and $b \in \mathcal H_j(\mathcal K)$, and we have $\alpha(a \otimes b)= a\wedge b$ in $\mathcal H_{i+j}(\mathcal K) $.  
As noted in Proposition \ref{prop:tensor}, we have the equality $\mathcal{P}^S_R(z) = \mathcal{P}^{S_1}_{R_1}(z) \mathcal{P}^{S_2}_{R_2}(z)$ of Poincar\'e series, which tells us that $$\dim_{\mathbb{k}}(\mathcal H_k(\mathcal K)) =\sum_{i+j=k} \dim_{\mathbb{k}}(\mathcal H_i(\mathcal K_1)) \cdot \dim_{\mathbb{k}}(\mathcal H_j(\mathcal K_2)).$$
Therefore, it follows that if $\{a_{i1}, \ldots, a_{ir_i}\}$ is a $\mathbb{k}$-basis of $\mathcal H_i(\mathcal K_1)$ and $\{b_{j1}, \ldots, b_{js_j}\}$ is one for $\mathcal H_j(\mathcal K_2)$, then a $\mathbb{k}$-basis of $\mathcal H_k(\mathcal K)$ is given by 
$\bigcup\limits_{i+j=k}\{ a_{ip}\wedge b_{jq} \mid 1\leq p \leq r_i, 1 \leq q \leq s_j \}$. We have thus proved the following.
\begin{prop}\label{prop:disjoint-ideals-imply-nontrivial-product}
    Suppose that $I_1\subseteq S_1=\mathbb{k}[x_1, \ldots, x_n]$ and $I_2\subseteq S_2=\mathbb{k} [y_1, \ldots, y_m]$ are nonzero homogeneous ideals generated in degree $\geq 2$, $S=S_1 \otimes_{\mathbb{k}}S_2$, and $I=I_1S+I_2S$. Then the product on the Koszul homology $\mathcal H(\mathcal K(x_{1}, \ldots, x_{n}, y_1, \ldots, y_m; S/I))$ is nontrivial.
\end{prop}
\begin{rmk}{\rm 
    While Proposition \ref{prop:disjoint-ideals-imply-nontrivial-product}, together with the fact that the Koszul homology of Golod rings have trivial product, can be readily applied to prove the \textit{non-trivial} implication in the statement of Proposition~\ref{prop:tensor}, we would like to remark that the proof of Proposition \ref{prop:tensor} is still instructive. The extra factor $\mu(I_1)\mu(I_2)$ showing up in the proof is precisely equal to $\dim_{\mathbb{k}}(\mathcal H_1(\mathcal K_1))\cdot \dim_{\mathbb{k}}(\mathcal H_1(\mathcal K_2))$, which can be thought of as the \textit{first obstruction} to Golodness of $S/I$, as expected. 
}\end{rmk}

Proposition \ref{prop:disjoint-ideals-imply-nontrivial-product} readily implies statements such as if $I \subseteq \mathbb{k}[x_1, \ldots, x_n]$ is a monomial ideal that can be written as a sum of two monomial ideals whose generators do not share a common variable (e.g., if $I$ is a complete intersection of codimention at least two), then $S/I$ does not have a trivial product on its Koszul homology, and hence is not Golod.

The version of the Proposition \ref{prop:disjoint-ideals-imply-nontrivial-product} which will be using crucially in next section is as follows.
\begin{cor}\label{cor:nontrivial-product-if-det-and-disjoint}
    Let $I$ be $2$-subdeterminantal ideal of $X=(x_{ij})_{m \times n}$ and $S=\mathbb{k}[X]$. If $I$ can be written as a sum of two $2$-subdeterminantal ideals of $X$ whose generating minors do not share a common variable, then the product on the Koszul homology $\mathcal H(\mathcal K(x_{11}, \ldots, x_{mn}; S/I))$ is nontrivial.
\end{cor}

\section{Golod 2-Subdeterminantal Ideals}

This section is devoted to the proof of Theorem~\ref{intro-main-thm}. We begin with several preliminary lemmas.

\begin{lemma}\label{lem:2byn}
    Let $X = \begin{bmatrix}
    x_1 & x_2 & \cdots & x_n \\
    y_1 & y_2 & \cdots & y_n
\end{bmatrix}$ be a $2 \times n$ matrix of indeterminates, $S=\mathbb{k}[X]$, and $I$ be a nonzero $2$-subdeterminantal ideal of $X$. Then the following are equivalent:
\begin{enumerate}
    \item[{\rm (i)}] The product on $\mathcal H(\mathcal K(x_1, \ldots,x_n, y_1, \ldots, y_n; S/I))$ is trivial.
    \item[{\rm (ii)}] $I=I_2(Y)$ for some $2 \times \ell$ submatrix $Y$ of $X$.
\end{enumerate}
\end{lemma}
\begin{proof}
      (ii) $\Longrightarrow$ (i).\  When $I=I_2(Y)$ for some $2 \times \ell$ matrix $Y$, the Eagon--Northcott complex (see \cite{EN62}) gives a linear minimal free resolution of $I$ over $S$. From \cite[Theorem 4]{HRW99}, we know that if $I$ is componentwise linear and contains no linear forms, then $S/I$ is Golod.  Thus, by Remark \ref{rmk:Golod-def}(b), we get that the product on $\mathcal H(\mathcal K(x_1, \ldots,x_n, y_1, \ldots, y_n; S/I))$ is trivial.

       (i) $\Longrightarrow$ (ii).\  Suppose that the product on $\mathcal H(\mathcal K(x_1, \ldots,x_n, y_1, \ldots, y_n; S/I))$ is trivial. For the sake of contradiction, assume that $I$ is not of the form $I_2(Y)$. Then in view of Corollary \ref{cor:nontrivial-product-if-det-and-disjoint}, we get that there exist $1 \leq i<j<k\leq n$ such that exactly two out of the following three minors are minimal generators of $I$:
    $$ M_1=\begin{vmatrix}
x_i & x_j \\
y_i & y_j
\end{vmatrix}, \quad  M_2=\begin{vmatrix}
x_j & x_k \\
y_j & y_k
\end{vmatrix}, \quad  M_3=\begin{vmatrix}
x_i & x_k \\
y_i & y_k
\end{vmatrix}.
$$
Without loss of generality, suppose that $M_1$ and $M_2$ are minimal generators of $I$, and $M_3$ is not. 

 Recall that the Koszul complex $\mathcal K(x_1, \ldots, x_n, y_1, \ldots, y_n; S/I)$ is a graded complex with degree zero maps and the grading defined as in Remark \ref{rmk:grading}. Restrict the Koszul complex to the degrees $\leq r\epsilon_i+r\epsilon_j+r\epsilon_k$ for $r\gg 0$. Then the resulting complex is precisely the Koszul complex on the elements $x_i, x_j, x_j, y_i, y_j, y_k$ over $R'=\mathbb{k}[x_i, x_j, x_k, y_i, y_j, y_k]/(M_1, M_2 )$. 

 Note that $\mathbb{k}[x_i, x_j, x_k, y_i, y_j, y_k]/( M_1, M_2 )$ is a complete intersection, and hence Gorenstein. From the work of Avramov and Golod \cite{AG71}, it follows that the Koszul homology algebra of $S/I$ is a Poincar\'e algebra. Since $\mathcal H(\mathcal K(x_i, x_j, x_k, y_i, y_j, y_k; R'))$ has socle degree at least $2$ by Remark~\ref{rmk:det-ideal-height}, it has nontrivial product. Hence, $\mathcal H(\mathcal K(x_1, \ldots,x_n, y_1, \ldots, y_n; S/I))$ has nontrivial product, which is a contradiction. Thus, $I$ must be of the form $I_2(Y)$ for some $2 \times \ell$ submatrix $Y$ of $X$.
\end{proof}

The following is a crucial observation, which follows immediately from the proof of (i) $\Longrightarrow$ (ii) in the lemma above.

\begin{cor}\label{cor:two-rows}
    Let $X=(x_{ij})_{m \times n}$ be a matrix of indeterminates, $S=\mathbb{k}[X]$, and let $I$ be a $2$-subdeterminantal ideal of $X$. Given any $1\leq j_1< j_2\leq n$, let $I^{j_1, j_2}$ be the ideal generated by the generators of $I$ of the form $$\begin{vmatrix}
x_{i_1j_1} & x_{i_1j_2} \\
x_{i_2j_1} & x_{i_2j_2}
\end{vmatrix}.$$
If $\mathcal{H}(\mathcal K(x_{11},\ldots, x_{mn}; S/I))$ has trivial product, then $I^{j_1, j_2} = I_2(Y)$ for some $\ell  \times 2$ submatrix $Y$ of $$\begin{bmatrix}
    x_{1j_1} & \cdots & x_{mj_1} \\
    x_{1j_2} & \cdots & x_{mj_2}
\end{bmatrix}^T.$$ 

Replacing $X$ by $X^T$, it is immediate that an analogous statement holds for a choice of any two rows instead of columns.
\end{cor}

We next focus on the $3 \times 3$ case. 
\begin{lemma}\label{lem:3by3}
    Let $X=(x_{ij})_{3\times 3}$ be a matrix of indeterminates, $S=\mathbb{k}[X]$, and $I$ be a nonzero $2$-subdeterminantal ideal $X$.
Then the following are equivalent:
\begin{enumerate}
  \item[{\rm (i)}] The product on $\mathcal H(\mathcal K(x_{11},\ldots,x_{33}; S/I))$ is trivial.
  \item[{\rm (ii)}] $I=I_2(Y)$ for some $2\times \ell$ or $\ell\times 2$ submatrix $Y$ of $X$, where $\ell=2$ or $3$.
\end{enumerate}
\end{lemma}
\begin{proof}
   As noted previously, if (ii) holds, then $S/I$ is Golod, and hence (i) holds.

For the converse, suppose that  (i) holds. If $\mu(I)=1$, then $I$ is of the form $I_2(Y)$, and we are done. So, assume that $\mu(I) \geq 2$. As noted in the proof of Lemma~\ref{lem:2byn},
if $S/I$ is Gorenstein, then by \cite{AG71} and Remark~\ref{rmk:det-ideal-height}, the Koszul homology algebra of $S/I$
is a Poincar\'e algebra of socle degree at least $2$; and in this case, (i)
does not hold. Therefore, we may further assume that $S/I$ is not Gorenstein. 

If $\mu(I)\le 3$, then the only possibility for $I$ to fail to be Gorenstein (in fact, a complete intersection)
is that $I=I_2(Y)$ for some $2\times 3$ or $3\times 2$ submatrix $Y$ of $X$, which implies (ii).
Hence, we may assume that $\mu(I)\ge 4$.
 Now, consider the matrix 
\[
X=\begin{bmatrix}
x_{11} & x_{12} & x_{13} \\
x_{21} & x_{22} & x_{23} \\
x_{31} & x_{32} & x_{33}
\end{bmatrix}
\]
and let 
\[
\begin{aligned}
M_1 &= \begin{vmatrix} x_{11} & x_{12} \\ x_{21} & x_{22} \end{vmatrix}, &
M_2 &= \begin{vmatrix} x_{11} & x_{13} \\ x_{21} & x_{23} \end{vmatrix}, &
M_3 &= \begin{vmatrix} x_{12} & x_{13} \\ x_{22} & x_{23} \end{vmatrix}, \\[1em]
M_4 &= \begin{vmatrix} x_{11} & x_{12} \\ x_{31} & x_{32} \end{vmatrix}, &
M_5 &= \begin{vmatrix} x_{11} & x_{13} \\ x_{31} & x_{33} \end{vmatrix}, &
M_6 &= \begin{vmatrix} x_{12} & x_{13} \\ x_{32} & x_{33} \end{vmatrix}, \\[1em]
M_7 &= \begin{vmatrix} x_{21} & x_{22} \\ x_{31} & x_{32} \end{vmatrix}, &
M_8 &= \begin{vmatrix} x_{21} & x_{23} \\ x_{31} & x_{33} \end{vmatrix}, &
M_9 &= \begin{vmatrix} x_{22} & x_{23} \\ x_{32} & x_{33} \end{vmatrix}.
\end{aligned}
\]
If $\mu(I)=9$, then $S/I$ is Gorenstein by \cite[Corollary 8.9]{BV88}. Hence, we may assume that
$\mu(I)\le 8$. Without loss of generality, suppose that $M_9\notin I$. Then, applying
Corollary~\ref{cor:two-rows} to the submatrix of rows $2$ and $3$, we may assume that
$M_8\notin I$. Similarly, applying the corollary to the submatrix of columns $2$ and $3$,
we may assume that $M_6\notin I$.

Applying Corollary~\ref{cor:two-rows} to the submatrix consisting of rows $1$ and $3$,
we see that at most one of the elements $M_4$ and $M_5$ can be a generator of $I$. 
Similarly, applying the corollary to the submatrix of columns $1$ and $3$, at most one
of the elements $M_2$ and $M_5$ can be a generator of $I$. Thus, we have the following cases:

 $\bullet$ \textit{Case 1:}  $M_5$ is a generator of $I$.\\
    In this case, since $\mu(I)\ge 2$, we must have $I = ( M_1, M_3, M_5, M_7 )$. 
    However, applying Corollary~\ref{cor:two-rows} to the submatrix formed by rows $1$ and $2$,
    we see that condition (i) does not hold, which is impossible.
    
   $\bullet$ \textit{Case 2:} $M_5$ is not a generator of $I$.\\
    In this case, either $I = ( M_1, M_2, M_3, M_4, M_7)$, or $I$ is generated by a
    subset of $\{M_1, M_2, M_3, M_4, M_7\}$ with $4$ elements. 
    In the former case, $S/I$ is Gorenstein, a contradiction. In the latter case, applying
    Corollary~\ref{cor:two-rows} to the submatrix formed by rows $1$ and $2$ shows that $I$ contains $M_1, M_2, M_3$.
A similar argument applied to the submatrix formed by columns $1$ and $2$ implies $I$ contains $M_1, M_4, M_7$.
    Therefore, $I = ( M_1, M_2, M_3, M_4, M_7)$, which is impossible.

This shows that $I = I_2(Y)$ for some $2 \times \ell$ or $\ell \times 2$ submatrix $Y$ of $X$,
with $\ell = 2$ or $3$.
\end{proof}

\begin{cor}\label{cor:3byn}
    Let $X=(x_{ij})_{3\times n}$ be a matrix of indeterminates, $S=\mathbb{k}[X]$, and $I$ be a nonzero $2$-subdeterminantal ideal of
$X$.
Then the following are equivalent:
\begin{enumerate}
  \item[{\rm (i)}] The product on $\mathcal H(\mathcal K(x_{11},\ldots,x_{3n}; S/I))$ is trivial.
  \item[{\rm (ii)}] $I=I_2(Y)$ for some $2\times \ell$ or $\ell\times 2$ submatrix $Y$ of $X$ with $\ell \geq 2$.
\end{enumerate}
Replacing $X$ by $X^T$, we see that an analogous statemtent holds for $n \times 3$ matrix of indeterminates.
\end{cor} 
\begin{proof}
    As noted before, (ii) $\Longrightarrow$ (i) is true. Conversely, suppose that (i) holds. Recall that the Koszul complex $\mathcal K(x_{11},\ldots,x_{3n}; S/I)$ is graded with respect to the grading defined by $\deg(x_{ij})=\epsilon_j$. Thus, if $X^{j_1,j_2,j_3}$ is the submatrix of $X$ consisting of columns $j_1, j_2, j_3$, and if $I^{j_1,j_2,j_3}$ is the ideal of $S$ generated by the generating minors of $I$ coming from $X^{j_1,j_2,j_3}$, then by Lemma \ref{lem:3by3}, we see that $I^{j_1, j_2, j_3}=I_2(Y)$ for some $2\times \ell$ or $\ell \times 2$ submatrix $Y$ of $X^{j_1,j_2,j_3}$. 

    For the sake of contradiction, assume that (ii) does not hold. Then we must have $\mu(I)\geq 2$. Without loss of generality, let $M_1\coloneqq\begin{vmatrix}
        x_{11} & x_{12} \\ x_{21} & x_{22}
    \end{vmatrix} \in I$. We consider the following cases: 
    
   $\bullet$ \textit{Case 1:} $M_2\coloneqq\begin{vmatrix}
        x_{21} & x_{22} \\ x_{31} & x_{32}
    \end{vmatrix} \in I$.\\
    By our assumption and Corollary \ref{cor:nontrivial-product-if-det-and-disjoint}, $I$ must have a minor generator $M_3$ of the form $\begin{vmatrix}
        x_{i_1j} & x_{i_1k} \\ x_{i_2j} & x_{i_2k}
    \end{vmatrix}$, where $j \leq 2$ and $k >2$. But as mentioned above, taking $(j_1, j_2,j_3)=(1,2,k)$, we get a contradiction.

  $\bullet$  \textit{Case 2:} $M_2 \not\in I$. \\
    In view of Corollary \ref{cor:two-rows}, after permuting the columns of $X$ if necessary, we may let $\ell$ be the integer such that the ideal $I_{1,2}$ generated by generators of $I$ coming from the first two rows of $X$ equals $I_2(Y)$, where $Y= \begin{bmatrix} 
    x_{11} & \cdots & x_{1\ell} \\
    x_{21} & \cdots & x_{2\ell}
\end{bmatrix}$.

We claim that $I=I_2(Y)$. By our assumption and Corollary \ref{prop:tensor}, $I$ must have a minor generator $M_3\not\in I_2(Y)$ containing at least one, and at most two variables appearing in $Y$. If $M_3$ has one variable common with $Y$, then by permuting the columns of $X$ if necessary, we may assume that $M_3=\begin{vmatrix}
        x_{22} & x_{2k} \\ x_{32} & x_{3k}
    \end{vmatrix}$, for some $k >\ell$. But this gives a contradiction by taking $(j_1, j_2, j_3)=(1, 2, k)$. If $M_3$ has two variables common with $Y$, then again, after a permutation of columns if necessary, we may assume that $M_3=\begin{vmatrix}
        x_{22} & x_{23} \\ x_{32} & x_{33}
    \end{vmatrix}$. This again leads to a contradiction by choosing $(j_1,j_2, j_3)=(1,2,3)$.  
\end{proof}

We are now ready to prove our main theorem.

\textit{Proof of Theorem \ref{intro-main-thm}}.    

The implication (i) $\implies$ (iv) holds by Remark \ref{rmk:Golod-def}(c).

    From \cite[Theorem 4]{HRW99}, we know that if $I$ is componentwise linear and contains no linear forms, then $S/I$ is Golod. Since ideals with linear resolution are componentwise linear, (ii) $\implies$ (i) holds.

    When $I=I_2(Y)$ for some $2 \times \ell$ or $\ell \times2$ matrix $Y$, the Eagon--Northcott complex gives a linear minimal free resolution of $I$ over $S$. Thus, (iii) $\implies$ (ii) is true.

    To complete the proof, we now prove (iv) $\implies$ (iii). The idea of the proof is very much similar to that of Corollary \ref{cor:3byn}. Firstly, observe that if $\mu(I)\leq 1$, then there is nothing to show. So, assume $\mu(I)\geq 2$. Let $\ell$ be the largest integer such that there exits a $ 2 \times \ell$ or $\ell \times 2$ submatrix $Y$ of $X$, such that every $2 \times 2$ minor of $Y$ is a minimal generator of $I$. Without loss of generality, and if necessary, replacing $X$ with $X^T$, we may assume that 
    $$Y=\begin{bmatrix}
    x_{11} &\cdots & x_{1\ell} \\
    x_{21} &\cdots & x_{2\ell}
\end{bmatrix}.$$
We consider two cases. 

 $\bullet$ \textit{Case 1:}\ $\ell=2$.\\
By our choice of $\ell$, it follows that if $M$ is any $2 \times 2$ minor of $X$ which is a generator of $I$, then $M$ contains at most one variable from the variables appearing in $Y$. If there is no $M$ with a variable common with those in $Y$, then $I$ is of the form $I_2(Y)+J$, for some nonzero ideal $J$ having a generating set involving variables not appearing in $Y$. In this case, (iv) does not hold by Corollary \ref{cor:nontrivial-product-if-det-and-disjoint}. So, suppose that there is one common variable between $M$ and $Y$. Without loss of generality, let $M=\begin{vmatrix} x_{22} & x_{23} \\ x_{32} & x_{33} \end{vmatrix}$. Then applying Corollary \ref{cor:3byn} to the first three columns of $X$, we get that (iv) does not hold.

$\bullet$ \textit{Case 2:}\ $\ell \geq 3$.\\
By a similar analysis as in Case 1 above, given any $M$, it must intersect $Y$ in at most $2$ variables. If no $M$ intersects with $Y$, or if the intersection is a single variable, then as in Case 1, we see that (iv) does not hold. In case the intersection contains 2 variables, without loss of generality (or more precisely, permuting the columns, if necessary), we may assume that  $M= \begin{vmatrix} x_{21} & x_{22} \\ x_{31} & x_{32} \end{vmatrix}$. 
Again, applying Corollary \ref{cor:3byn} to the the first 3 columns of $X$, we see that (iv) does not hold. 

    This shows that $I$ must be equal to $I_2(Y)$, completing the proof. \hfill{} $\square$

\begin{rmk}\label{rmk:general_mn} {\rm
 In general, if $I$ is a $t$-subdeterminantal ideal of an $m \times n$ matrix $X$ of indeterminates, where $t, m, n\geq 3$, then the Golodness of $S/I$ does not imply that $I$ has a linear resolution over $S$. For instance, as the   \texttt{Macaulay2}~\cite{M2} computations below show, the ring $\mathbb{k} [X]/I$ is Golod when $I$ generated by any three of the four $3 \times 3$ minors of a $3 \times 4$ matrix $X$ of indeterminates, but in this case, $I$ does not have a linear resolution. The same example also tells us that for $t \geq 3$, the Golodness of $S/I$ does not force $I$ to be of the form $I_t(Y)$ for some $t \times \ell$ or $\ell \times t$ submatrix $Y$ of $X$.
}\end{rmk}    

\vspace{12pt}

\begin{lstlisting}[
    basicstyle=\ttfamily\scriptsize,
    numbers=none,
    frame=single,
    framerule=0.5pt,
    xleftmargin=5pt,
    xrightmargin=5pt,
    aboveskip=2pt,
    belowskip=2pt,
    lineskip=-1pt,
    breaklines=true,
    showstringspaces=false,    mathescape=true
]
i1 : loadPackage "DGAlgebras";

i2 : m = 3;

i3 : n = 4;

i4 : S = QQ[x_(1,1)..x_(m,n)];

i5 : X = transpose genericMatrix(S, n, m);

o5 : Matrix ${S^3}$ <--- ${S^3}$

i6 : mins = flatten entries gens minors(3, X);

i7 : I = ideal(mins#0, mins#1, mins#2);

o7 : Ideal of S

i8 : betti res I

o8 =

            0 1 2 3
total:      1 3 3 1
    0:      1 . . .
    1:      . 3 . .
    2:      . . 3 1

o8 : BettiTally

i9 : R = S/I;

i10 : isGolod R

o10 = true
\end{lstlisting}

\vspace{12pt}    
Theorem \ref{intro-main-thm} and Remark \ref{rmk:general_mn} above suggest us to ask the following.     

    \begin{question}
       Let $t,m,n \geq 3$, $X=(x_{ij})_{m \times n}$ be a matrix of indeterminates, and $S=\mathbb{k}[X]$. Suppose that $I$ is a $t$-subdeterminantal ideal of $X$. Is it true that $S/I$ is Golod if and only if the product on the Koszul homology of $S/I$ is trivial?
    \end{question}

    \begin{question}
         Let $t,m,n \geq 3$, $X=(x_{ij})_{m \times n}$ be a matrix of indeterminates, and $S=\mathbb{k}[X]$. Let $I$ be a  $t$-subdeterminantal ideals of $X$. When is $S/I$ Golod?
    \end{question}
(Generalized) binomial edge ideals form a special case of (generalized) determinantal facet ideals (see \cite{AV22, EHHT13}). 
In recent years, the study of generalized determinantal facet ideals, much like that of generalized binomial edge ideals, has attracted significant interest. 
The above question can therefore be reformulated as asking which generalized determinantal facet ideals associated with pure simplicial complexes of dimension $(t-1)$ are Golod.

In the literature, determinantal ideals arising from other classes of matrices of indeterminates, such as symmetric or Hankel matrices, have also been studied. 
One may ask similar questions about the Golod property of the rings defined by these ideals. 
Another natural direction is to study analogous questions for permanental ideals.
In general, permanental ideals exhibit a completely different and non-uniform behavior when compared with determinantal ideals. 
It would be interesting to know what happens to the Golod property for permanental ideals.



\begin{thebibliography}{10}

\bibitem{Ah19}
R.~Ahangari~Maleki.
\newblock The {G}olod property for powers of ideals and {K}oszul ideals.
\newblock {\em J. Pure Appl. Algebra}, 223(2):605--618, 2019.

\bibitem{AV22}
A.~Almousa and K.~VandeBogert.
\newblock Determinantal facet ideals for smaller minors.
\newblock {\em Arch. Math. (Basel)}, 118(3):247--256, 2022.

\bibitem{ANFY17}
N.~Altafi, N.~Nemati, S.~A. Seyed~Fakhari, and S.~Yassemi.
\newblock Free resolution of powers of monomial ideals and {G}olod rings.
\newblock {\em Math. Scand.}, 120(1):59--67, 2017.

\bibitem{Av98}
L.~L. Avramov.
\newblock Infinite free resolutions.
\newblock In {\em Six lectures on commutative algebra ({B}ellaterra, 1996)}, volume 166 of {\em Progr. Math.}, pages 1--118. Birkh\"auser, Basel, 1998.

\bibitem{AG71}
L.~L. Avramov and E.~S. Golod.
\newblock The homology of algebra of the {K}oszul complex of a local {G}orenstein ring.
\newblock {\em Mat. Zametki}, 9:53--58, 1971.

\bibitem{BEI18}
H.~Baskoroputro, V.~Ene, and C.~Ion.
\newblock Koszul binomial edge ideals of pairs of graphs.
\newblock {\em J. Algebra}, 515:344--359, 2018.

\bibitem{BJ07}
A.~Berglund and M.~J\"ollenbeck.
\newblock On the {G}olod property of {S}tanley-{R}eisner rings.
\newblock {\em J. Algebra}, 315(1):249--273, 2007.

\bibitem{BH93}
W.~Bruns and J.~Herzog.
\newblock {\em Cohen-{M}acaulay rings}, volume~39 of {\em Cambridge Studies in Advanced Mathematics}.
\newblock Cambridge University Press, Cambridge, 1993.

\bibitem{BV88}
W.~Bruns and U.~Vetter.
\newblock {\em Determinantal rings}, volume 1327 of {\em Lecture Notes in Mathematics}.
\newblock Springer-Verlag, Berlin, 1988.

\bibitem{CE56}
H.~Cartan and S.~Eilenberg.
\newblock {\em Homological algebra}.
\newblock Princeton University Press, Princeton, NJ, 1956.

\bibitem{DS22}
H.~Dao and A.~De~Stefani.
\newblock On monomial {G}olod ideals.
\newblock {\em Acta Math. Vietnam.}, 47(1):359--367, 2022.

\bibitem{DS16}
A.~De~Stefani.
\newblock Products of ideals may not be {G}olod.
\newblock {\em J. Pure Appl. Algebra}, 220(6):2289--2306, 2016.

\bibitem{EN62}
J.~A. Eagon and D.~G. Northcott.
\newblock Ideals defined by matrices and a certain complex associated with them.
\newblock {\em Proc. Roy. Soc. London Ser. A}, 269:188--204, 1962.

\bibitem{EHHT13}
V.~Ene, J.~Herzog, T.~Hibi, and F.~Mohammadi.
\newblock Determinantal facet ideals.
\newblock {\em Michigan Math. J.}, 62(1):39--57, 2013.

\bibitem{EHHQ14}
V.~Ene, J.~Herzog, T.~Hibi, and A.~A. Qureshi.
\newblock The binomial edge ideal of a pair of graphs.
\newblock {\em Nagoya Math. J.}, 213:105--125, 2014.

\bibitem{Fr18}
R.~Frankhuizen.
\newblock {$A_\infty$}-resolutions and the {G}olod property for monomial rings.
\newblock {\em Algebr. Geom. Topol.}, 18(6):3403--3424, 2018.

\bibitem{Go62}
E.~S. Golod.
\newblock Homologies of some local rings.
\newblock {\em Dokl. Akad. Nauk SSSR}, 144:479--482, 1962.

\bibitem{M2}
D.~R. Grayson and M.~E. Stillman.
\newblock Macaulay2, a software system for research in algebraic geometry.
\newblock Available at \url{http://www.math.uiuc.edu/Macaulay2/}.

\bibitem{GTSW16}
J.~Grbi\'c, T.~Panov, S.~Theriault, and J.~Wu.
\newblock The homotopy types of moment-angle complexes for flag complexes.
\newblock {\em Trans. Amer. Math. Soc.}, 368(9):6663--6682, 2016.

\bibitem{HerzogHibiZheng}
J.~Herzog, T.~Hibi, and X.~Zheng.
\newblock Monomial ideals whose powers have a linear resolution.
\newblock {\em Math. Scand.}, 95(1):23--32, 2004.

\bibitem{HH13}
J.~Herzog and C.~Huneke.
\newblock Ordinary and symbolic powers are {G}olod.
\newblock {\em Adv. Math.}, 246:89--99, 2013.

\bibitem{HRW99}
J.~Herzog, V.~Reiner, and V.~Welker.
\newblock Componentwise linear ideals and {G}olod rings.
\newblock {\em Michigan Math. J.}, 46(2):211--223, 1999.

\bibitem{IK18}
K.~Iriye and D.~Kishimoto.
\newblock Golodness and polyhedral products for two-dimensional simplicial complexes.
\newblock {\em Forum Math.}, 30(2):527--532, 2018.

\bibitem{IK23}
K.~Iriye and D.~Kishimoto.
\newblock Golod and tight 3-manifolds.
\newblock {\em Algebr. Geom. Topol.}, 23(5):2191--2212, 2023.

\bibitem{IK24}
K.~Iriye and D.~Kishimoto.
\newblock Tight complexes are {G}olod.
\newblock {\em Int. Math. Res. Not. IMRN}, (8):6471--6495, 2024.

\bibitem{JK25}
A.~Jayanthan and A.~Kumar.
\newblock (generalized) binomial edge ideals and their regularity.
\newblock {\em arXiv preprint arXiv:2508.12607}, 2025.

\bibitem{Jo09}
D.~A. Jorgensen.
\newblock On tensor products of rings and extension conjectures.
\newblock {\em J. Commut. Algebra}, 1(4):635--646, 2009.

\bibitem{Ka16}
L.~Katth\"an.
\newblock The {G}olod property for {S}tanley-{R}eisner rings in varying characteristic.
\newblock {\em J. Pure Appl. Algebra}, 220(6):2265--2276, 2016.

\bibitem{Ka17}
L.~Katth\"an.
\newblock A non-{G}olod ring with a trivial product on its {K}oszul homology.
\newblock {\em J. Algebra}, 479:244--262, 2017.

\bibitem{LMMP26}
A.~LaClair, M.~Mastroeni, J.~McCullough, and I.~Peeva.
\newblock Koszul binomial edge ideals.
\newblock {\em Forum Math. Sigma}, 14:Paper No. e3, 2026.

\bibitem{Se65}
J.-P. Serre.
\newblock {\em Alg\`ebre locale. {M}ultiplicit\'es}, volume~11 of {\em Lecture Notes in Mathematics}.
\newblock Springer-Verlag, Berlin-New York, 1965.
\newblock Cours au Coll\`ege de France, 1957--1958, r\'edig\'e{} par Pierre Gabriel, Seconde \'edition, 1965.

\bibitem{FW14}
S.~A. Seyed~Fakhari and V.~Welker.
\newblock The {G}olod property for products and high symbolic powers of monomial ideals.
\newblock {\em J. Algebra}, 400:290--298, 2014.

\bibitem{Va22}
K.~VandeBogert.
\newblock Products of ideals and {G}olod rings.
\newblock {\em Proc. Amer. Math. Soc.}, 150(8):3345--3356, 2022.

\bibitem{Zo24}
J.~Zoromski.
\newblock Monomial cycles in koszul homology.
\newblock {\em arXiv preprint arXiv:2409.07583}, 2024.

\end{thebibliography}
\end{document}